\documentclass[a4paper,11pt, leqno]{article}
\usepackage[latin1]{inputenc}
\usepackage{hyperref}
\usepackage{articlemacro}
\numberwithin{equation}{section}

\DeclareMathOperator{\Sph}{Sph}

\begin{document}

\title{Spherical Spaces}

\author{Torsten Wedhorn}

\maketitle

%------------------------------------------------------------------

\noindent{\scshape Abstract.\ }
The notion of a spherical space over an arbitrary base scheme is introduced as a generalization of a spherical variety over an algebraically closed field. It is studied how the sphericity condition behaves in families. In particular it is shown that sphericity of subgroup schemes is an open and closed condition over arbitrary base schemes generalizing a result by Knop and R\"ohrle. Moreover spherical embeddings are classified over arbitrary fields generalizing and simplifying results by Huruguen.

\medskip


%==================================================================

\section*{Introduction}

Spherical varieties over an algebraically closed field are normal varieties with the action of a reductive group $G$ such that a Borel subgroup of $G$ acts with an open dense orbit. They generalize several important classes of varieties: projective homogeneous $G$-varieties, symmetric varieties, and toric varieties. In this paper we generalize the notion of a spherical variety to algebraic spaces over an arbitrary base scheme $S$. If $G$ is a reductive group over $S$, then a \emph{spherical $G$-space over $S$} is a flat separated algebraic space of finite presentation over $S$ with a $G$-action such that all geometric fibers are spherical varieties (see below for precise definitions). A flat and finitely presented subgroup scheme $H$ of $G$ is called \emph{spherical} if $G/H$ is a spherical $G$-space. The first main result is the following (Corollary~\ref{SphericityOpen} and Theorem~\ref{SphericalOpenClosed}).

\begin{intro-theorem}
\begin{assertionlist}
\item
The property for a flat finitely presented $G$-space with normal geometric fibers to be spherical is open and constructible on the base scheme.
\item
The property for a flat finitely presented subgroup scheme $H$ of $G$ to be spherical is open and closed on the base scheme.
\end{assertionlist}
\end{intro-theorem}

The second result generalizes a result of Knop and R\"ohrle (\cite[Theorem~3.4]{KnRo}) who considered the case that $S$ is the spectrum of a Dedekind domain and that $G$ is a split reductive group scheme\footnote{In \cite{KnRo} Knop and R\"ohrle state already in a footnote that their result is ``surely valid in greater generality'' and here this prediction is shown to be true.}. It is proved by reducing to this case and invoking their result.

In the second part of the paper we study spherical spaces over arbitrary fields. Over an algebraically closed field spherical varieties are classified in two steps.
\begin{assertionlist}
\item
One first classifies spherical embeddings of a homogeneous spherical $G$-variety by combinatorial objects called colored fans (generalizing the classification of toric varieties by fans). This is done by the theory of Luna, Vust, and Knop (\cite{LunaVust} in characteristic $0$, \cite{Knop_LunaVust} for arbitrary characteristic). 
\item
Then one classifies spherical subgroups $H$ of $G$. This classification is now complete in characteristic $0$ by work of Luna \cite{Luna}, Losev \cite{Losev}, Cupit-Foutou \cite{CuFo}, and Bravi and Pezzini \cite{BP1}, \cite{BP2}, \cite{BP3}. But it is still open in positive characteristic.
\end{assertionlist}
We have nothing new to say for the second point (except for some trivialities about classification of forms in Section~\ref{CLFORM}). Here we generalize the first classification result to spherical spaces over arbitrary fields. There are two technical ingredients here. We first show that Knop's methods of classification by colored fans over algebraically closed fields can easily be generalized to the case of separably closed field (Theorem~\ref{ClassifySepClosed}). Then we use Galois descent to classify spherical embeddings of a spherical homogeneous $G$-space in terms of Galois invariant colored fans (Theorem~\ref{ClassField}). Similar results have already been obtained by Huruguen in \cite{Hur}. The results here are slightly more general as we consider also non-perfect base fields and as we do not assume that a homogeneous space has a rational point. They are also simplified by working systematically with algebraic spaces instead of schemes which allows us to get rid of Huruguen's technical condition~(ii).

\tableofcontents

\bigskip\bigskip

\noindent\textbf{Notation}

\smallskip

\noindent\textit{Algebraic Spaces}.
An \emph{algebraic space} over a scheme $S$ is an fppf-sheaf $X$ on the category of $S$-schemes such that the diagonal $X \to X \times_S X$ is representable by a morphism of schemes and such that there exists a surjective \'etale morphism $\Xtilde \to X$, where $\Xtilde$ is an $S$-scheme.

Hence the notion of algebraic space is the one in the sense of \cite[Definition~Tag~025Y]{Stacks}. All occurring algebraic spaces will be quasi-separated over some scheme and therefore they will be Zariski locally quasi-separated. Hence they are reasonable (and decent) in the sense of \cite[Tag~03I8]{Stacks} by \cite[Tag~03JX]{Stacks}.

Moreover all Zariski-local results of references such as \cite{LM_Stacks} or \cite{Rom_Comp}, where algebraic spaces are assumed to be quasi-separated, are still valid.

\smallskip

\noindent\textit{Fibers and Geometric Fibers}.
For each point $x$ of the underlying topological space of a decent algebraic space $X$ (we simply write $x \in X$) there exists a field $k$ and a monomorphism $\Spec(k) \to X$, and this monomorphism is unique up to unique isomorphism (\cite[Tag 03K4]{Stacks}). The field $k$ is called the \emph{residue field of $x$} and denoted by $\kappa(x)$.

If $X$ is an algebraic space over some algebraic space $S$ and $T \to S$ is a morphism of algebraic spaces, then we write $X_T$ instead of $X \times_S T$. If $T = \Spec(A)$ is an affine scheme, we also write $X_A$. For $s \in S$ we denote by $X_s$ the fiber $X_{\kappa(s)}$.

A geometric point is a morphism $\Spec k \to S$, where $k$ is an algebraically closed field. For $s \in S$ we write $\sbar = \Spec(\kappa(s)^{a})$, where $\kappa(s)^a$ is an algebraic closure of $\kappa(s)$. If $X \to S$ is an algebraic space over $S$, then $X_{\sbar} := X \times_S \sbar$ is called \emph{geometric fiber} of $X$ in $s$.

\smallskip

\noindent\textit{Subspaces}.
An immersion (resp.~an open immersion, resp.~a closed immersion) of algebraic spaces is a morphism of algebraic spaces that is representable and an immersion (resp.~an open immersion, resp.~a closed immersion) (\cite[Tag 03HB]{Stacks}). Two immersions whose target is an algebraic space $X$ define the same \emph{subspace} of $X$ if each factors through the other (necessarily uniquely as immersions are monomorphisms in the category of algebraic spaces). An immersion of algebraic spaces induces a locally closed embedding on the underlying topological spaces (\cite[Tag 04CD]{Stacks}). Conversely, if $X$ is an algebraic space and $T$ is a locally closed subset of the underlying topological space of $X$, then there exists a unique reduced subspace $Z \mono X$ whose underlying topological space is $T$ (\cite[Tag 06EC]{Stacks}).

\smallskip

\noindent\textit{Algebraic Groups}.
If $k$ is a field, an \emph{algebraic group} over $k$ is by definition a group scheme of finite type over $k$ (not necessarily smooth).

\paragraph*{Acknowledgements}
I am grateful to Jochen Heinloth for several useful discussions and to the referee for pointing out that several results on algebraic spaces needed more detailed explanations.

%------------------------------------------------------------------

\section{Spherical varieties}

Let $k$ be an algebraically closed field and let $G$ be a reductive algebraic group over $k$. A normal connected separated $G$-scheme $X$ of finite type over $k$ is called \emph{spherical} or \emph{$G$-spherical} if it satisfies the following equivalent conditions.
\begin{equivlist}
\item\label{Sphi}
There exists a Borel subgroup $B$ of $G$ which has an open orbit in $X$. 
\item\label{Sphii}
There exists a Borel subgroup $B$ of $G$ and a point $x \in X(k)$ such that $\dim \Stab_B(x) = \dim B - \dim X$.
\item\label{Sphiii}
There exists a Borel subgroup $B$ of $G$ such that $B$ acts with finitely many orbits on $X$.
\end{equivlist}
The equivalence of \ref{Sphi} and \ref{Sphii} is clear and the equivalence of \ref{Sphi} and \ref{Sphiii} follows from by \cite[Corollary~2.6]{Knop_OrbitBorel}. As all Borel subgroups are conjugate, all conditions hold for every Borel subgroup of $G$ if they hold for one Borel subgroup. Note that we assume that by definition a connected space is non-empty implying that spherical varieties are non-empty.

We need the fact that ``having finitely many orbits'' is stable under change of the base field. Here we will use this fact only for schemes over algebraically closed fields (see Corollary~\ref{AlgClosedExt} below). But for later reference this result is formulated more generally than needed here. We fix the following notation. Let $\kappa$ be a field, let $H$ be a smooth algebraic group over $\kappa$, and let $X$ be an algebraic space of finite type over $\kappa$ with an action of $H$. 

For schemes of finite type over an algebraically closed field with an action by a smooth algebraic group it is equivalent having finitely many orbits or having finitely many invariant reduced subschemes. In this case the orbits are the minimal invariant reduced subspaces in the following sense. 

\begin{definition}\label{MinimalSubscheme}
An $H$-invariant subspace $Y$ of $X$ is called \emph{minimal} if there exists no proper non-empty $H$-invariant subspace of $Y$.
\end{definition}

If $Y$ is any $H$-invariant subspace, then the underlying reduced subspace $Y_{\rm red}$ is $H$-invariant because $H$ is smooth and hence geometrically reduced. This shows in particular that minimal $H$-invariant subspaces are reduced.

Algebraic spaces of finite type over an arbitrary field $\kappa$ with an action by a smooth algebraic group might not have any orbit (defined over the base field) because they might not have $\kappa$-rational points. But if they have a rational point than a minimal $H$-invariant subspace is the same as an orbit:
 
\begin{lemma}\label{MinimalSubspaceQuot}
An $H$-invariant subspace $Y$ that has a $\kappa$-rational point $x$ is minimal if and only if it is $H$-equivariantly isomorphic to $H/\Stab_H(x)$.

In particular $Y$ is a smooth scheme over $\kappa$ in this case.
\end{lemma}

\begin{proof}
Indeed, the condition is certainly sufficient. Conversely suppose that $Y$ is minimal. Let $f\colon H \to Y$ be the orbit morphism of $\kappa$-spaces given by $h \sends h\cdot x$. Then $f$ is generically flat because $Y$ is reduced (\cite[Tag 06QS]{Stacks}) and hence flat by homogenity. Therefore $f$ is open (\cite[Tag 042S]{Stacks}). Its image is an $H$-invariant subspace of $Y$, hence $f$ is surjective because of the minimality of $Y$. Therefore $f$ is faithfully flat and hence an epimorphism in the category of fppf-sheaves over $\kappa$. The induced morphism $H/\Stab_H(x) \to Y$ is therefore a monomorphism and an epimorphism of fppf-sheaves. Hence it is an isomorphism.
\end{proof}

We now show that having finitely many minimal invariant subspaces is invariant under change of base field.

\begin{lemma}\label{FiniteOrbit}
Let $\kappa$ be a field, let $H$ be a smooth algebraic group over $\kappa$ and let $X$ be an algebraic space over $\kappa$ of finite type with $H$-action. Let $K$ be a field extension of $\kappa$. Then $X$ has only finitely many minimal $H$-invariant subspaces if and only if $X_K$ has only finitely many minimal $H_K$-invariant subspaces.
\end{lemma}

\begin{proof}
The condition is clearly sufficient. Hence let us assume that $X$ has only finitely many munimal $H$-invariant subspaces. Passing to a finite extension $\kappa'$ of $\kappa$ and replacing $K$ by a field extension containning both $\kappa'$ and $K$ we may assume that every minimal $H$-invariant subspace $Y$ of $X$ has a $\kappa$-rational point. Then such a subspace $Y$ is isomorphic to a quotient of $H$ by Lemma~\ref{MinimalSubspaceQuot}. As the formation of quotients commutes with base change, such a minimal $H$-invariant smooth subscheme stays minimal $H$-invariant and smooth after the base change $\kappa \to K$. Therefore there are only finitely many minimal $H_K$-invariant subspaces of $X_K$.
\end{proof}

\begin{corollary}\label{AlgClosedExt}
Let $K$ be an algebraically closed extension of the algebraically closed field $k$ and let $X$ be $G$-scheme over $k$. Then $X$ is spherical if and only if $X_K$ is spherical.
\end{corollary}

\begin{proof}
Note first that $X$ is of finite type if and only if $X_K$ is of finite type. Moreover, if $X_K$ is connected (resp.~normal), then clearly $X$ is connected (resp.~normal). The converse holds by \cite[(4.5.1) and (6.7.7)]{EGA4}. Hence we are done by Lemma~\ref{FiniteOrbit} applied to a Borel subgroup $H$ of $G$.
\end{proof}

%---------------------------------------------------------------------

\section{Spherical spaces}

Let $S$ be a scheme and let $G$ be a reductive group scheme over $S$ (i.e., $G$ is a smooth affine group scheme over $S$ whose geometric fibers are connected reductive groups).

\begin{definition}\label{DefSphericalSpace}
An algebraic space $X$ over $S$ with an action by $G$ is called \emph{spherical} or \emph{$G$-spherical} if it is flat, separated, and of finite presentation over $S$ and if for all $s \in S$ the geometric fiber $X_{\sbar}$ is a spherical $G_{\sbar}$-variety.
\end{definition}

As spherical varieties are by definition non-empty, the structure morphism $X \to S$ of a spherical space is always surjective and hence faithfully flat.

\begin{remark}\label{RemSphericalSpace}
If $S = \Spec(k)$ for a field $k$, then a $G$-spherical space over $k$ is an algebraic space $X$ over $k$ with $G$-action such that $X_K$ is a $G_K$-spherical variety for some (or equivalently by Lemma~\ref{AlgClosedExt}) for any algebraically closed extension $K$ of $k$. There exists then a finite separable field extension $k'$ such that $X_{k'}$ is a scheme (\cite[Tag~0B84]{Stacks}).
\end{remark}

\begin{remark}\label{PermSpherical}
Let $X$ be an algebraic space over $S$ with $G$-action. Let $S' \to S$ be a morphism of schemes.
\begin{assertionlist}
\item\label{PermSpherical1}
If $X$ is $G$-spherical, then the $S'$-space $X_{S'}$ is $G_{S'}$-spherical.
\item\label{PermSpherical2}
Let $S' \to S$ be an fpqc-covering and assume that $X_{S'}$ is $G_{S'}$-spherical. Then $X$ is $G$-spherical. 
\end{assertionlist}
\end{remark}

\begin{proof}
The property for a morphism of algebraic spaces to be flat (resp.~separated, resp.~of finite presentation) is stable under base change (\cite[Tags 03MO,03KL,03XR]{Stacks}) and fpqc local on the base (\cite[Tags 041W,0421,041V]{Stacks}).

Moreover, the geometric fibers of $X_{S'}$ are base changes of geometric fibers of $X$. If $S' \to S$ is an fpqc-covering, then $S' \to S$ is surjective, and for every geometric fiber $X_{\sbar}$ of $X$ there is a geometric fiber of $X_{S'}$ which is the base change of $X_{\sbar}$. Hence we conclude by Corollary~\ref{AlgClosedExt}.
\end{proof}

%---------------------------------------------------------------------

\section{Openness of sphericity}

For the next result recall that there is the notion of constructible set in an arbitrary scheme $S$.

\begin{definition}\label{DefConst}
A subset $C$ of a scheme $S$ such that for every (equivalently, for one) open covering $(U_i)_i$ of $S$ by affine subschemes the intersections $S \cap U_i$ are for all $i$ finite unions of sets of the form $V_i \cap (U_i \setminus W_i)$, where $V_i,W_i \subseteq U_i$ are open quasi-compact subsets. 
\end{definition}

Here we follow the convention of \cite[\S7]{EGA1} (and not of \cite{EGA4}, where a constructible set in our sense is called locally constructible). If $S$ is noetherian, then then $C \subseteq S$ is constructible if and only if $C$ is a finite union of locally closed subsets.

\begin{theorem}\label{SphericalConst}
Let $S$ be a scheme, let $X$ be an algebraic space that is of finite presentation over $S$ and that has geometrically irreducible fibers. Let $H$ be a group scheme of finite presentation over $S$ that acts on $X$. Define
\[
C := \sett{s \in S}{there exists an open dense $H_{\sbar}$-orbit in $X_{\sbar}$}
\]
\begin{assertionlist}
\item\label{SphericalConst1}
The subset $C$ is constructible in $S$.
\item\label{SphericalConst2}
Suppose in addition that $H$ and $X$ are flat over $S$. Then $C$ is also open.
\end{assertionlist}
\end{theorem}

\begin{proof}
Both assertions are Zariski locally on $S$ and hence we may assume that $S$ is affine. Hence $X$ is quasi-compact and quasi-separated because the structure morphism $\pi\colon X \to S$ is of finite presentation. In particular there exists an \'etale quasi-compact cover $\xi\colon \Xtilde \to X$ by an affine scheme $\Xtilde$.

Choose an integer $N$ such that all fibers of $H$ and of $X$ in points of $S$ have dimension $\leq N$ (\cite[14.105]{GW} applied to $H$ and $\Xtilde$ because one has $\dim(X_s) = \dim(\Xtilde_s)$ for all $s \in S$). For $x \in X$ we obtain an action of the fiber $H_{\pi(x)}$ of $H$ on the fiber $X_{\pi(x)}$ of $X$ and we denote by $\Stab_H(x)$ the stabilizer of the canonical morphism $\Spec \kappa(x) \to X_{\pi(x)}$ in $H_{\pi(x)}$, which is a closed subgroup scheme of $H \times_S \Spec(\kappa(x))$. Define
\begin{equation}\label{EqDefHomogGeneral}
X^0 := \set{x \in X}{\dim \Stab_H(x) = \dim H_{\pi(x)} - \dim X_{\pi(x)}}.
\end{equation}
Then $\pi(X^0) = C$ and hence it suffices by Chevalley's theorem (\cite[(7.2.9)]{EGA1} if $X$ is a scheme, \cite[Th\'eor\`eme 5.9.4]{LM_Stacks} in general) to show that $X^0$ is a constructible subset of $X$.

Fix an integer $d$ with $d \leq N$. Let $E_d$ be the set of $x \in X$ such that $\dim H_{\pi(x)} - \dim X_{\pi(x)} = d$. Let $F_d$ be the set of $x \in X$ such that $\dim \Stab_H(x) = d$. Then 
\[
X^0 = \bigcup_{d \leq N}(E_d \cap F_d).
\]
Hence to show that $X^0$ is constructible it suffices to show that $E_d$ and $F_d$ are constructible.

Now $E_d$ is constructible by \cite[(9.9.1)]{EGA4} (the result generalizes to algebraic spaces by applying it to $\Xtilde$ and again using that $\dim(X_s) = \dim(\Xtilde_s)$ for all $s \in S$). To see that $F_d$ is constructible, we define an algebraic space $\Hcal$ over $X$ by the cartesian diagram
\[\xymatrix{
\Hcal \ar[r]^g \ar[d] & X \ar[d]^{\Delta} \\
H \times_S X \ar[r]^{f} & X \times_S X,
}\]
where $f$ is the morphism $(h,x) \sends (h\cdot x,x)$ and where $\Delta$ is the diagonal. Then $\Stab_H(x)$ is the fiber $\Hcal_x$ of $g$ in $x$. Hence $F_d$ is constructible by \cite[10.96]{GW}. Here we again use that this result generalizes to algebraic spaces: We apply \cite[10.96]{GW} to the morphism of schemes $\tilde\Hcal := \Hcal \times_X \Xtilde \to \Xtilde$ to see that $\Ftilde_d := \set{\xtilde \in \Xtilde}{\dim(\tilde\Hcal_{\xtilde}) = d}$ is constructible in $\Xtilde$. As $\xi$ is locally of finite presentation and quasi-compact, $F_d = \xi(\Ftilde_d)$ is constructible by Chevalley's theorem.

It remains to show that $C$ is open if $H$ and $X$ satisfy the hypotheses in \ref{SphericalConst2}. As $X$ is flat and locally of finite presentation over $S$, the morphism $\pi$ is open. Hence it suffices to show that $X^0$ is open. Moreover the additional hypotheses on $H$ and $X$ imply that the maps $s \sends \dim(H_s)$ and $s \sends \dim(X_s)$ are locally constant on $S$. Hence we may assume that $e := \dim H_s - \dim X_s$ is independent of $s$. Then $X^0$ is the set, where $\Stab_H(x) = \Hcal_x$ has the minimal possible dimension $e$. Hence it is open by semi-continuity of fiber dimension (where we can argue as above by replacing $\Hcal \to X$ with the morphism of schemes $\tilde{\Hcal} \to \Xtilde$ and using that $\xi$ is open because it is \'etale).
\end{proof}

\begin{corollary}\label{SphericityOpen}
Let $S$ be a scheme, let $X$ be an algebraic space which is separated, flat, and of finite presentation over $S$ with geometrically normal and geometrically integral fibers. Let $G$ be a reductive group scheme over $S$ that acts on $X$. Then there exists an open constructible subscheme $U$ of $S$ such that a morphism of schemes $f\colon T \to S$ factors through $U$ if and only if $X_T$ is a spherical $G_T$-space.
\end{corollary}

\begin{proof}
We have to show that the functor on $S$-schemes
\[
(f\colon T \to S) \quad\sends \begin{cases}
\{f\},&\text{if $X_T$ is a spherical $G_T$-space}; \\
\emptyset,&\text{otherwise}
\end{cases}
\]
is representable by an open constructible subscheme of $S$. As this can be shown locally for the \'etale topology on $S$, we may assume that $G$ is a split reductive group scheme \cite[Exp. XXII, Cor.~2.4]{SGA3} and in particular that there exists a Borel subgroup scheme $B$ of $G$. Applying Theorem~\ref{SphericalConst} with $H = B$ we obtain an open constructible subscheme $U$ of $S$ such that $U = \sett{s \in S}{$X_{\sbar}$ is $G_{\sbar}$-spherical}$.

This represents the functor because being spherical is stable under base change and local for the fpqc topology. Indeed, if $f\colon T \to S$ factors through $U$, then $X_T = (X_U)_T$ is $G_T$-spherical by Remark~\ref{PermSpherical}~\ref{PermSpherical1}. Conversely, if $f$ does not factor through $U$, then there exists $t \in T$ such that $f(t) \in S \setminus U$. As $X_{f(t)}$ is not $G_{f(t)}$-spherical, $(X_T)_t = (X_{f(t)})_t$ is not $G_t$-spherical by Remark~\ref{PermSpherical}~\ref{PermSpherical2}. Hence $X_T$ cannot be $G_T$-spherical.
\end{proof}

\begin{remark}\label{HomogSpherical}
Let $G$ be a reductive group over a scheme $S$ and let $X$ be a spherical $G$-space. Then the open subspace $X^0$ of $X$ defined in \eqref{EqDefHomogGeneral} is given by the property that for every $s \in S$ the preimage of $X^0$ in the geometric fiber $X_{\sbar}$ is the unique open dense $G_{\sbar}$-orbit. In particular it is faithfully flat over $S$ and in each geometric fiber schematically dense.
\end{remark}

We conclude this section by studying limits of spherical spaces. Let $(S_i)_{i\in I}$ be a filtered projective system of quasi-compact quasi-separated schemes with affine transition morphism $\pi_{ij}\colon S_j \to S_i$. Let $S$ be its limit (which exists in the category of schemes). Assume that $I$ contains a smallest element $0$, let $G_0$ be a group scheme of finite presentation over $S_0$ and set $G_i := G_0 \times_{S_0} S_i$ and $G := G_0 \times_{S_0} S$. We first recall the following general fact that follows from \cite[Tag~07SK]{Stacks}.

\begin{remark}\label{Approx}
Let $X$ be an algebraic space of finite presentation over $S$ with an action by $G$. Then there exists $i \in I$, an algebraic space $X_i$ over $S_i$ with $G_i$-action, and a $G$-equivariant isomorphism $\alpha\colon X_i \times_{S_i} S \iso X$ of spaces over $S$.

Moreover, if $(i',X'_{i'},\alpha')$ is a second such triple, then there exists a $j \geq i,i'$ and a $G_j$-equivariant isomorphism $\tau_j\colon X_i \times_{S_i} S_j \iso X'_{i'} \times_{S_{i'}} S_j$ such that $\alpha' \circ (\tau_j \times \id_S) = \alpha$.
\end{remark}

Now we deduce from Theorem~\ref{SphericalConst} that being spherical is compatible with filtered limits in the following sense.

\begin{corollary}\label{SphericalInductive}
Let $G_0$ be reductive and let $X_0$ be an algebraic space of finite presentation of $S_0$ with a $G_0$-action. Suppose that $X := X_0 \times_{S_0} S$ is a $G$-spherical space. Then there exists an $i \in I$ such that $X_i := X_0 \times_{S_0} S_i$ is a $G_i$-spherical space.
\end{corollary}

\begin{proof}
There exists $i_0 \in I$ such that $X_{i_0} \to S_{i_0}$ is separated and faithfully flat (\cite[Section~Tag~084V]{Stacks}). For $i \geq i_0$ let $C_i \subseteq S_i$ be the set of $s \in S_i$ such that $(X_i)_s$ is geometrically normal and geometrically integral. Then $C_i$ is constructible in $S_i$ by Lemma~\ref{GeomConst} below.

For $j \geq i \geq i_0$ one has $C_j = \pi_{ij}^{-1}(C_i)$ and the inverse image of $C_i$ in $S$ is equal to $S$ because $X$ has geometrically normal and geometrically integral fibers. By \cite[10.57]{GW} the set of (open) constructible subsets in $S$ is the filtered colimit of the sets of (open) constructible subsets of $S_i$. Hence there exists a $i_1 \geq i_0$ such that $C_{i_1} = S_{i_1}$, i.e., $X_{i_1}$ has geometrically normal and geometrically integral fibers.

Now we can use the same argument with the spherical locus $U_i \subseteq S_i$ of $X_i$ for $i \geq i_1$ as given by Corollary~\ref{SphericityOpen}. We find $i_2 \geq i_1$ such that $X_{i_2}$ is $G_{i_2}$-spherical.
\end{proof}

\begin{lemma}\label{GeomConst}
Let $S$ be a scheme and let $X \to S$ be an algebraic space of finite presentation over $S$. Let $C(X)$ be the set of $s \in S$ such that $X_{\sbar}$ is geometrically normal (resp.~geometrically integral). Then $C(X)$ is a constructible subset of $S$.
\end{lemma}

\begin{proof}
Consider the property ``geometrically normal''. Let $\Xtilde \to X$ be an \'etale surjective morphism where $\Xtilde$ is a scheme, then $C(X) = C(\Xtilde)$. Hence $C(X)$ is constructible in $S$ by \cite[(9.9.5)]{EGA4}.

The constructibility for the property ``geometrically integral'' is shown in \cite{Rom_Comp}.
\end{proof}

%---------------------------------------------------------------------

\section{Closedness of sphericity for subgroups}

We now consider the special case where $X$ is a quotient of $G$. Let $S$ be a scheme, let $G$ be a reductive group scheme over $S$, and let $H$ be a closed subgroup scheme of $G$ that is flat and of finite presentation over $S$. Define the \emph{spherical locus}
\[
\Sph_G(H) := \sett{s \in S}{$G_{\sbar}/H_{\sbar}$ is $G_{\sbar}$-spherical}.
\]
We call $H$ a \emph{spherical subgroup of $G$} if $\Sph_G(H) = S$.

\begin{proposition}\label{CharSphericalSubgroups}
In the situation above the following conditions are equivalent.
\begin{equivlist}
\item\label{CharSphericalSubgroups1}
The subgroup scheme $H$ is a spherical subgroup of $G$.
\item\label{CharSphericalSubgroups2}
The algebraic space $G/H$ is $G$-spherical.
\item\label{CharSphericalSubgroups3}
The conjugation action of $H$ on the scheme $\Bor(G)$ of Borel subgroups of $G$ has in all geometric fibers a dense orbit.
\end{equivlist}
\end{proposition}

Recall that $\Bor(G)$ is the smooth projective $S$-scheme of finite presentation with geometrically integral fibers whose $T$-valued points for some $S$-scheme $T$ is the set of Borel subgroups $G_T$ (\cite[Exp. XXII~5.8.3]{SGA3}). 

\begin{proof}
For all $s \in S$ some Borel subgroup $B$ of $G_{\sbar}$ acts on $G_{\sbar}/H_{\sbar}$ with open dense orbit if and only if $H_{\sbar}$ acts with open dense orbit on $G_{\sbar}/B \cong \Bor(G_{\sbar})$. This shows the equivalence of \ref{CharSphericalSubgroups1} and \ref{CharSphericalSubgroups3}. As the formation of the quotient $G/H$ commutes with arbitrary base change and in particular with passage to geometric fibers, it is clear that \ref{CharSphericalSubgroups2} implies \ref{CharSphericalSubgroups1}.

The converse follows from facts about quotients recalled in the appendix.
\end{proof}

\begin{remark}\label{NonSmoothSpherical}
If $S$ is of characteristic $0$ (i.e., a scheme over $\Spec \QQ$), then every flat group scheme locally if finite presentation over $S$ is automatically smooth. But in general a spherical subgroup as defined above is not necessarily smooth over the base scheme. For instance let $S$ be the spectrum of an algebraically closed field of characteristic $p > 0$ and let $G = \GG_{m,S}$. Then the scheme $\mu_p$ of $p$-th roots of unity is a spherical subgroup of $G$ in the above sense.
\end{remark}

\begin{theorem}\label{SphericalOpenClosed}
$\Sph_G(H)$ is open and closed in $S$.
\end{theorem}

In particular we see that if $S$ is connected and there exists a point $t \in S$ such that $H_{\tbar}$ is a spherical subgroup of $G_{\tbar}$, then $H_{\sbar}$ is a spherical subgroup of $G_{\sbar}$ for all $s \in S$. The theorem generalizes a result of Knop and R\"ohrle (\cite[Theorem~3.4]{KnRo}) who considered the case that $S$ is the spectrum of a Dedekind domain and that $G$ is split (i.e., $G$ has a split maximal torus $T$). In fact, we will reduce to this case using Corollary~\ref{SphericityOpen}.

\begin{proof}
To prove Theorem~\ref{SphericalOpenClosed} we can work locally for the fpqc topology on $S$ by Remark~\ref{PermSpherical}~\ref{PermSpherical2}. Hence we may assume that $G$ is a split reductive group scheme.

We consider the action of $H$ by conjugation on the scheme of Borel subgroups $\Bor(G)$, a smooth and projective scheme over $S$. Then Theorem \ref{SphericalConst} shows that $\Sph_G(H)$ is open and constructible in $S$\footnote{One could have chosen a Borel subgroup of $G$ (possible because $G$ is split) and then apply Theorem \ref{SphericalConst} to the $B$-space $G/H$.}.

It remains to prove that $\Sph_G(H)$ is closed. We may assume that $S = \Spec A$ is affine. By writing $A$ as a filtered inductive limit of finitely generated $\ZZ$-algebras, we may assume by Corollary~\ref{SphericalInductive} that $A$ is of finite type over $\ZZ$ in particular noetherian. Moreover, as we already know that $\Sph_G(H)$ is constructible, it suffices to show that $\Sph_G(H)$ is stable under specialization (see e.g. \cite[10.17]{GW}).

Let $x,y \in S$ be two points such that $x$ is a specialization of $y$. By \cite[15.7]{GW} there exists a discrete valuation ring and a morphism of schemes $g\colon \Spec R \to S$ such that $g(s) = x$ and $g(\eta) = y$, where $s$ (resp.~$\eta$) is the special (resp.~generic) point of $\Spec R$. Hence base change to $\Spec R$ allows us to reduce to the case that $S$ is the spectrum of a discrete valuation ring and that $x$ is the special point and $y$ is the generic point. This is the case considered in Theorem~3.4 of \cite{KnRo}, which shows that in this case $x \in \Sph_G(H)$ if and only if $y \in \Sph_G(H)$.
\end{proof}

%------------------------------------------------------------------

\section{Spherical embeddings}

Let $S$ be a scheme and let $G$ be a smooth group scheme of finite presentation over $S$. 

\begin{definition}\label{DefHomog}
An algebraic space $X$ with $G$-action is called a \emph{homogeneous} if $X$ is faithfully flat, of finite presentation, and separated over $S$ with geometrically reduced fibers and if $G(\kappa(\sbar))$ acts transitively on $X(\kappa(\sbar))$ for every geometric point $\sbar$ of $S$. 
\end{definition}

Using the same arguments as in Remark~\ref{PermSpherical} one sees that if $S' \to S$ is a morphism of schemes and $X$ is a homogeneous $G$-space, then $X_{S'}$ is a homogeneous $G_{S'}$-space. The converse holds if $S' \to S$ is a fpqc covering.

\begin{lemma}\label{HomogLocallyQuot}
A $G$-space $X$ over $S$ is homogeneous if and only if there exists an fppf-covering $S' \to S$, a closed subgroup scheme $H \subseteq G_{S'}$ flat and of finite presentation over $S'$, and a $G_{S'}$-equivariant isomorphism $G_{S'}/H \iso X_{S'}$ of algebraic spaces over $S'$.
\end{lemma}

The proof will show that a homogeneous $G$-space $X$ is of the form $G/H$ if and only if $X$ has a section over $S$.

\begin{proof}
If $X \cong G/H$ for a closed subgroup scheme $H$ of $G$, then $X$ is clearly a homogeneous $G$-space by the facts recalled in the appendix. Hence the sufficiency follows from descent.

Conversely, let $X$ be a homogeneous $G$-space. As $X$ is faithfully flat and of finite presentation over $S$, it admits locally for the fppf topology a section. Hence, we may assume that $X$ has a section $x\colon S \to X$. Let $H$ be the stabilizer of this section. By \cite[II, \S1, Th\'eor\`eme 3.6]{DeGa}, $H$ is a closed subgroup scheme of $G$. Let $\pi\colon G \to X$, $g \sends g\cdot x$ be the orbit morphism corresponding to $x$. Then $\pi$ is surjective, because it is surjective by hypothesis on geometric fibers. As all geometric fibers of $X$ are reduced, $\pi$ is generically flat on geometric fibers. By homogeneity it is faithfully flat on geometric fibers and hence $\pi$ itself is flat by the fiber criterion of flatness. Hence $\pi$ is an epimorphism for the fppf topology and therefore induces a $G$-equivariant isomorphism $G/H \iso X$.

It remains to show that $H$ is flat and of finite presentation over $S$. Consider the cartesian diagram
\[\xymatrix{
G \times_S H \ar[r]^-a \ar[d]_p & G \ar[d]^{\pi} \\
G \ar[r]^{\pi} & X,
}\]
where $a$ is the action of $H$ on $G$ by right multiplication and where $p$ is the first projection. Now $\pi$ is of finite presentation over $S$ because $G$ and $X$ are both of finite presentation over $S$ and flat. Hence $p$ is flat and of finite presentation. Therefore $H \to S$ is flat and of finite presentation by fpqc-descent via the quasi-compact faithfully flat morphism $G \to S$.
\end{proof}

\begin{remark}\label{HomogSmooth}
As $G$ is smooth over $S$, any quotient $G/H$ as above is smooth (see Appendix). As smoothness can be checked fppf-locally, one sees that a homogeneous $G$-space is automatically smooth over $S$.
\end{remark}

Over fields every homogeneous space is automatically a scheme:

\begin{lemma}\label{HomogeneousScheme}
Let $S = \Spec(k)$ for a field $k$. Then every homogenous $G$-space $X$ is a smooth quasi-projective scheme over $k$.
\end{lemma}

\begin{proof}
Let $\kbar$ be an algebraic closure of $k$. Then $X_{\kbar} \cong G_{\kbar}/H$ by Lemma~\ref{HomogLocallyQuot} for some algebraic subgroup $H$ of $G_{\kbar}$. In particular $X_{\kbar}$ is a smooth quasi-projective scheme over $\kbar$. As $X_{\kbar}$ is quasi-projective, any finite set of $\kbar$-valued points is contained in an open affine subscheme. Hence $X$ is a scheme by \cite[Tag~0B88]{Stacks}. It is smooth and quasi-projective by fpqc descent for field extensions.
\end{proof}

\begin{definition}\label{DefSphericalEmb}
Let $S$ be a scheme, let $G$ be a reductive group scheme over $S$ and let $X^0$ be a spherical homogeneous $G$-space. Then a \emph{spherical embedding of $X^0$} is a spherical $G$-space $X$ over $S$ together with an open $G$-equivariant immersion $i\colon X^0 \to X$.

Let $X^0$ and $Y^0$ be a spherical homogeneous $G$-spaces, let $(X,i)$ and $(Y,j)$ be spherical embeddings of $X^0$ and $Y^0$, respectively. A \emph{morphism $(X^0,X,i) \to (Y^0,Y,j)$ of spherical embeddings} is a $G$-equivariant morphism $\varphi^0\colon X^0 \to Y^0$ over $S$ such that there exists a $G$-equivariant morphism $\varphi\colon X \to Y$ with $\varphi \circ i = j \circ \varphi^0$.
\end{definition}

Note that $\varphi^0$ is automatically surjective. If $(X,i)$ is a spherical embedding of $X^0$, then the $i(X^0)$ is open and on all geometric fibers schematically dense in $X$. As spherical spaces are separated, this shows that $\varphi$ is uniquely determined by $\varphi^0$ if it exists.

\begin{remark}\label{BCMorph}
Let $X^0$ be a spherical homogenous $G$-space over $S$ and let $S' \to S$ be a morphism of schemes. Then the base change functor $(\ )_{S'}$ yields a functor from the category of $G$-spherical embeddings of $X^0$ to the category of $G_{S'}$-spherical embeddings of $(X_0)_{S'}$.
\end{remark}

\begin{remark}\label{SphericalEmbExist}
Let $G$ be a reductive group scheme over $S$ and let $X$ be a spherical $G$-space. Then there exists a homogenous spherical $G$-space $X_0$ and a spherical embedding $X^0 \to X$. Indeed, let $X^0$ be the open subspace of $X$ defined in \eqref{EqDefHomogGeneral}. Then $X^0$ is $G$-invariant and a homogeneous spherical $G$-space by Remark~\ref{HomogSpherical}. 
\end{remark}

%---------------------------------------------------------------------

\section{Classification of spherical embeddings over algebraically closed fields}\label{CLASSALG}

In the remaining parts of the paper we will give a classification of spherical embeddings over arbitrary fields by reducing it to the classification over algebraically closed fields by Knop (\cite{Knop_LunaVust}, see also \cite{LunaVust} in characteristic $0$). Hence let us first recall this classification. We follow \cite{Knop_LunaVust} almost verbatim.

Until the end of this section $k$ will denote an algebraically closed field, $G$ will be a reductive group over $k$ and $X^0$ will be a homogeneous spherical $G$-space (hence the choice of some $x \in X^0(k)$ would yield a $G$-equivariant isomorphism $X^0 \to G/H$ for some spherical subgroup $H \subseteq G$).

\subsection{Data attached to homogeneous spherical spaces}\label{DataGH}

We choose a Borel subgroup $B$ of $G$ and a maximal torus $T$ of $B$. Denote by $K(X^0)^{(B)}$ the set of $B$-eigenvectors in the function field $K(X^0)$, i.e. the subset of elements $0 \ne v \in K(X^0)$ such that $b\cdot v = \chi_v(b)v$ for some character $\chi_v$ of $B$. This is a subgroup of $K(X^0)^{\times}$ and
\[
\chi\colon K(X^0)^{(B)} \to X^*(B) = X^*(T), \qquad v \sends \chi_v
\]
is a homomorphism of groups, where $X^*(H)$ denotes the group of characters of an algebraic group $H$. Set
\begin{align*}
\Lambda := \Lambda_{(G,X^0)} := &\Im(\chi) \subseteq X^*(B), \\
\Qcal := \Qcal_{(G,X^0)} := &\Hom_{\ZZ}(\Lambda,\QQ) = \Lambda\vdual_{\QQ}.
\end{align*}
Via bi-duality we will consider $\Lambda$ as a subgroup of $\Qcal\vdual$, the dual of the finite-dimensional $\QQ$-vector space $\Qcal$.

As usual, $\Lambda$ (and hence $\Qcal$) do not depend on the choice of $(B,T)$ up to {\em unique} isomorphism: if $(B',T')$ is a second Borel pair, then there exists $g \in G(k)$ such that $B = gBg^{-1}$ and $T = gTg^{-1}$ and conjugation with $g$ yields an isomorphism $X^*(B) \to X^*(B')$ that is independent of the choice of $g$ because for any other choice $g'$ one has $g' = gt$ for some $t \in T(k)$.

Any $\QQ$-valued valuation $v$ on $K(X^0)$ that is trivial on $k$ induces a homomorphism $\Lambda \to \QQ$ and hence an element $\rho_v \in \Qcal$. Let
\[
\Vcal := \Vcal_{(G,X^0)}
\]
be the set of $\QQ$-valued $G$-invariant valuations on $K(X^0)$ that are trivial on $k$. Then the map
\[
\Vcal \to \Qcal, \qquad v \sends \rho_v
\]
is injective and we identify $\Vcal$ with its image in $\Qcal$.

Let $D$ be a prime divisor on $X^0$, where by a divisor we always mean a Weil divisor. Then the local ring at the generic point of $D$ is a discrete valuation ring and we obtain a valuation $v_D$ of $K(X^0)$ that is trivial on $k$. Set
\[
\Dcal := \Dcal_{(G,X^0)} := \{\text{$B$-stable prime divisors in $X^0$}\}.
\]
This is a finite set. We obtain a map
\[
\rho\colon \Dcal \to \Qcal, \qquad D \sends \rho_{v_D}.
\]
Again, $\Dcal$ and $\rho$ depend on the choice of $(B,T)$ only up to unique isomorphism: if $(B',T') = g(B,T)g^{-1}$ for some $g \in G(k)$ then conjugation with $g$ yields a bijection between the sets $\Dcal$ defined with $B$ and $B'$ which is independent of the choice of $g$.

These construction are functorial in the following sense. Let $X_1^0$ and $X^0_2$ be spherical homogeneous $G$-spaces and let $\varphi^0\colon X^0_1 \to X^0_2$ be a $G$-equivariant morphism. Then $\varphi^0$ induces an injection $\Lambda_{(G,X^0_2)} \to \Lambda_{(G,X^0_1)}$ and hence a surjective $\QQ$-linear map
\[
\varphi^0_*\colon \Qcal_{(G,X^0_1)} \to \Qcal_{(G,X^0_2)}
\]
such that $\varphi^0_*(\Vcal_{(G,X^0_1)}) = \Vcal_{(G,X^0_2)}$. Moreover, by sending a divisor to its schematic image we obtain a map
\[
\varphi^0_*\colon \Dcal_{\varphi^0} := \sett{D \in \Dcal_{(G,X^0_1)}}{$\varphi^0(D)$ is not dense in $X^0_2$} \to \Dcal_{(G,X^0_2)}.
\]

\subsection{Colored Fans}

We now recall the notion of a colored cone and a colored fan for the finite-dimensional $\QQ$-vector space $\Qcal$, the subset $\Vcal \subseteq \Qcal$, the finite set $\Dcal$, and the map $\rho\colon \Dcal \to \Qcal$.

\begin{definition}\label{DefColoredCone}
A \emph{colored cone for $(G,X^0)$} (or \emph{for $(\Qcal,\Vcal,\Dcal, \rho)$}) is a pair $(C,F)$ where $C \subseteq \Qcal_{(G,X^0)}$ is a cone (i.e., closed under addition and multiplication with $\QQ_{\geq0}$) and $F \subseteq \Dcal_{(G,X^0)}$ is a subset satisfying the following properties.
\begin{definitionlist}
\item[(CC1)]
$C$ is generated as a cone by $\rho(F)$ and finitely many elements of $\Vcal_{(G,X^0)}$.
\item[(CC2)]
The relative interior $C^\circ$ of $C$ (i.e., $C$ minus all proper faces) meets $\Vcal_{(G,X^0)}$.
\end{definitionlist}
A colored cone $(C,F)$ is called \emph{strictly convex} if $C$ is strictly convex (i.e. $C \cap (-C) = \{0\}$) and $0 \notin \rho(F)$. 
\end{definition}

Let $(C,F)$ be a colored cone for $(G,X^0)$ and let $C_0$ be a face of $C$ whose relative interior $C_0^\circ$ meets $\Vcal$. Set $F_0 := F \cap \rho^{-1}(C_0)$. Then $(C_0,F_0)$ is again a colored cone and such a pair is called a \emph{face of $(C,F)$}. If $(C,F)$ is strictly convex so is $(C_0,F_0)$.

\begin{definition}\label{DefColordFan}
A \emph{colored fan for $(G,X^0)$} (or \emph{for $(\Qcal,\Vcal,\Dcal, \rho)$}) is a nonempty finite set $\Fscr$ of colored cones for $(G,X^0)$ satisfying the following properties.
\begin{definitionlist}
\item[(CF1)]
Every face of a colored cone in $\Fscr$ is again in $\Fscr$.
\item[(CF2)]
For every $v \in \Vcal$ there exists at most one $(C,F) \in \Fscr$ with $v \in C^{\circ}$.
\end{definitionlist}
A colored fan $\Fscr$ is called \emph{strictly convex} if all elements of $\Fscr$ are strictly convex.
\end{definition}

We also have the notion of a morphisms of colored fans.

\begin{definition}\label{DefMorphColored}
Let $X^0_1$ and $X^0_2$ be spherical homogeneous $G$-spaces and let $\varphi^0\colon X^0_1 \to X^0_2$ be a $G$-equivariant morphism.
\begin{assertionlist}
\item
Let $(C_1,F_1)$ and $(C_2,F_2)$ be colored cones for $(G,X^0_1)$ and $(G,X^0_2)$, respectively. Then $\varphi^0$ is called a \emph{morphism $(C_1,F_1) \to (C_2,F_2)$ of colored cones} if $\varphi^0_*(C_1) \subseteq C_2$ and $\varphi^0(F_1 \cap \Dcal_{\varphi}) \subseteq F_2$.
\item
Let $\Fscr_1$ and $\Fscr_2$ be colored fans for $(G,X^0_1)$ and $(G,X^0_2)$, respectively. Then $\varphi^0$ is called a \emph{morphism $\Fscr_1 \to \Fscr_2$ of colored fans} if for every $(C_1,F_1) \in \Fscr_1$ there exists $(C_2,F_2) \in \Fscr_2$ such that $\varphi^0$ is a morphism $(C_1,F_1) \to (C_2,F_2)$ of colored cones.
\end{assertionlist}
\end{definition}

\subsection{Colored fans attached to spherical embeddings}\label{CCSpherical}

Let $(X,i)$ be a spherical embedding of $X^0$. For every $G$-orbit $Y$ in $X$, let ${\rm Prim}_Y(X)$ be the set of prime divisors on $X$ containing $Y$. Set
\begin{align*}
\Bcal_Y(X) &:= \sett{v_D \in \Vcal}{$D \in {\rm Prim}_Y(X)$ and $G$-stable},\\
\Fcal_Y(X) &:= \sett{D \cap X^0 \in \Dcal}{$D \in {\rm Prim}_Y(X)$ and $B$-stable but not $G$-stable},\\
\Ccal_Y(X) &:= \langle \Bcal_Y(X), \rho(\Fcal_Y(X))\rangle,
\end{align*}
where the last line means that $\Ccal_Y(X)$ is the cone in $\Qcal_{(G,X^0)}$ generated by $\Bcal_Y(X)$ and $\rho(\Fcal_Y(X))$. Then
\[
\Fscr(X,i) := \sett{(\Ccal_Y(X),\Fcal_Y(X))}{$Y \subseteq X$ is a $G$-orbit}
\]
is a strictly convex colored fan. Now the main result of \cite{Knop_LunaVust} is the following.

\begin{theorem}\label{ClassSphericalAlg}
Let $G$ be a reductive group over an algebraically closed field $k$ and let $X^0$ be a homogeneous spherical $G$-space. The constructions above yield an equivalence $(X,i) \sends \Fscr(X,i)$ between the category of spherical embeddings of $X^0$ and the category of colored fans for $(G,X^0)$.
\end{theorem}

\begin{remark}\label{RemClassAlg}
Let $G$ be a reductive group over $k$. The functoriality arguments in \cite[\S4]{Knop_LunaVust} show that if $X^0$ and $Y^0$ are homogeneous spherical $G$-spaces and $(X,i)$ and $(Y,j)$ are spherical embeddings of $X^0$ and $Y^0$ respectively, then a $G$-equivariant morphism $\varphi^0\colon X^0 \to Y^0$ is a morphism of spherical embeddings $(X^0,X,i) \to (Y^0,Y,j)$ if and only if $\varphi^0$ is a morphism of colored fans $\Fscr(X,i) \to \Fscr(Y,j)$.
\end{remark}

%---------------------------------------------------------------------

\section{Classification of spherical embeddings over separably closed fields}\label{CLASSSEP}

In this section we show that the classification of spherical embeddings over algebraically closed fields extends to separably closed fields $k$. Of course we may assume that the characteristic of $k$ is $p > 0$. Fix an algebraic closure $\kbar$ of $k$. 

By Remark~\ref{RemSphericalSpace} every spherical space over a separably closed field is a scheme. Hence in this section we will work only with schemes.

The field extension $k \mono \kbar$ is purely inseparable. Let us first recall some facts about such extensions.

\begin{remark}\label{PassageInsep}
Let $G$ be a smooth algebraic group over $k$ and let $X$ be any $k$-scheme of finite type with $G$-action. Let $\pi\colon X_{\kbar} \to X$ be the projection morphism.
\begin{assertionlist}
\item\label{PassageInsep1}
The morphism $\pi$ is a universal homeomorphism (\cite[5.46]{GW}). In particular $Y \sends \pi^{-1}(Y)_{\rm red}$ yields a bijection between reduced subschemes of $X$ and reduced subschemes of $X_{\kbar}$. Moreover, $Y$ is closed (resp.~irreducible, resp.~a prime Weil divisors) if any only if $\pi^{-1}(Y)_{\rm red}$ is. The subscheme $Y$ is $G$-invariant if and only if $\pi^{-1}(Y)_{\rm red}$ is $G_{\kbar}$-invariant.
\item\label{PassageInsep2}
If $X$ is geometrically reduced, then $X(k)$ is dense in $X$ (\cite[6.21]{GW}). In particular, every homogeneous $G$-space over $k$ (Definition~\ref{DefHomog}) is of the form $G/H$ for some algebraic subgroup $H$ of $G$. 
\item\label{PassageInsep3}
If $Y'$ is a $G_{\kbar}$-orbit in $X_{\kbar}$. Then by \ref{PassageInsep1} there exists a unique reduced $G$-invariant subscheme $Y$ of $X$ such that $\pi^{-1}(Y)_{\rm red} = Y'$. This is clearly a minimal $G$-invariant subscheme and we see that the bijection in \ref{PassageInsep1} yields a bijection between minimal $G$-invariant subschemes of $X$ and $G_{\kbar}$-orbits of $X_{\kbar}$.
\item\label{PassageInsep4}
If $X$ is geometrically integer over $k$, then $K(X_{\kbar}) = K(X) \otimes_k \kbar$ and the extension of function field $K(X) \subseteq K(X_{\kbar})$ is purely inseparable. In particular, any $\QQ$-valued valuation $v$ on $K(X)$ extends uniquely to a $\QQ$-valued valuation $\vbar$ on $K(X_{\kbar})$ (\cite[VI~\S8]{BouAC}). Moreover, $v$ is trivial on $k$ (resp.~$G$-invariant) if and only if $\vbar$ is trivial on $\kbar$ (resp.~$G_{\kbar}$-invariant).
\item\label{PassageInsep5}
If $G$ is reductive over $k$, then $G$ is split (\cite[Exp. XXII, Cor.~2.4]{SGA3}).
\end{assertionlist}
\end{remark}

We now classify spherical embeddings of spherical homogeneous spaces $X^0$ over the separably closed field $k$. By Remark~\ref{PassageInsep}~\ref{PassageInsep2} $X^0$ is of the form $G/H$ for some spherical subgroup $H$. The following result reduces the classification over $k$ to the known classification over $\kbar$.

\begin{theorem}\label{ClassifySepClosed}
Let $G$ be a reductive group over a separably closed field $k$ and let $X^0$ be a homogeneous spherical $G$-space. Then $X \sends X_{\kbar}$ yields a bijection between isomorphism classes of $G$-spherical embeddings of $X^0$ and isomorphism classes of $G_{\kbar}$-spherical embeddings of $X^0_{\kbar}$.
\end{theorem}

The proof is essentially by going through the proof of the classification of spherical embeddings over algebraically closed fields as explained in \cite{Knop_LunaVust} and showing that all arguments work -- after possible minor modifications -- over separably closed fields as well.

\begin{proof}
\proofstep{(i)}
We first show that all the constructions in Subsection~\ref{DataGH} and Subsection~\ref{CCSpherical} make also sense over a separably closed field and that one obtains the same invariants for $G$ and $X^0$ as for $G_{\kbar}$ and $X^0_{\kbar}$. Let $B$ be a Borel subgroup of $G$ and $T$ a maximal torus of $B$ (which exist by Remark~\ref{PassageInsep}~\ref{PassageInsep5}). Then we can identify $X^*(B) = X^*(B_{\kbar})$. Moreover it follows from Remark~\ref{PassageInsep}~\ref{PassageInsep4} that $K((X^0)_{\kbar})^{(B_{\kbar})} = K(X^0)^{(B)} \otimes_k \kbar$ and hence $\Lambda_{(G,X^0)} = \Lambda_{(G_{\kbar},X^0_{\kbar})}$. One obtains the $\QQ$-vector space $\Qcal_{(G,X^0)}$ and one has $\Qcal_{(G,X^0)} = \Qcal_{(G_{\kbar},X^0_{\kbar})}$. One defines $\Vcal_{(G,X^0)}$, $\Dcal_{(G,X^0)}$, and the map $\rho$ as in Subsection~\ref{DataGH}. Then point \ref{PassageInsep4} (resp.~\ref{PassageInsep1}) of Remark~\ref{PassageInsep} yields an identification $\Vcal_{(G,X^0)} = \Vcal_{(G_{\kbar},X^0_{\kbar})}$ (resp.~$\Dcal_{(G,X^0)} = \Dcal_{(G_{\kbar},X^0_{\kbar})}$) such that the map $\rho$ is the same.

For every $G$-spherical embedding $X$ of $X^0$ and every minimal $G$-invariant subscheme $Y$ of $X$ one defines the sets $\Bcal_Y(X)$, $\Fcal_Y(X)$, $\Ccal_Y(X)$ as in Subsection~\ref{CCSpherical}. Then one has with the previous identifications $\Bcal_{(Y_{\kbar})_{\red}}(X_{\kbar}) = \Bcal_Y(X)$ and similarly for the other sets.

\proofstep{(ii)}
We remark that it suffices to show the theorem for simple spherical embeddings, i.e., for spherical embeddings that have only one closed minimal $G$-invariant subscheme. The reason is that the gluing explained in \cite[Theorem~3.3]{Knop_LunaVust} carries over to the case of a separably closed base field because of Remark~\ref{PassageInsep}~\ref{PassageInsep1}.

\proofstep{(iii)}
Let $X$ be a simple spherical embedding of $G/H$ and $Y$ be the closed minimal $G$-invariant subscheme of $X$. Let us show that the colored cone $(\Ccal_Y(X),\Fcal_Y(X))$ determines the isomorphism class of $X$.

The same argument as in \cite[Lemma~2.4]{Knop_LunaVust} shows that $\Bcal_Y(X)$ can be recovered from $\Ccal_Y(X)$ and $\Fcal_Y(X)$. It uses \cite[\S1]{Knop_LunaVust} whose arguments go through for a separably closed base field except for the following modifications.
\begin{bulletlist}
\item
In \cite[Theorem~1.1]{Knop_LunaVust} one finds $f'$ a priori only in $k_1[X]^{(B)}$ where $k_1$ is an extension of $k$ in $\kbar$ such that there exists an integer $e \geq 1$ with $(k_1)^{p^e} \subseteq k$. Replacing $f'$ by $(f')^{p^e}$ one concludes that Theorem 1.1 also holds over separably closed fields.
\item
In the proof of \cite[Corollary~1.7]{Knop_LunaVust} one uses the fact that for every reductive group $G$ there exists a reductive group $\Gtilde$ and a central isogeny $\Gtilde \to G$ such that $\Pic(\Gtilde) = 1$. This holds over abritrary fields for example by \cite[Prop.~3.1]{CT_Res}.
\end{bulletlist}
It remains to show that $X$ is uniquely determined by $(\Bcal_Y(X),\Fcal_Y(X))$. This can be proved verbatim as in the proof of \cite[Theorem~2.3]{Knop_LunaVust} again using that all results of \cite[\S1]{Knop_LunaVust} are valid over separably closed base fields.

\proofstep{(iv)}
It remains to show that every strictly convex colored cone for $(G,X^0)$ comes from a simple spherical embedding. Here the arguments in the proof of \cite[Theorem~3.1]{Knop_LunaVust} can be used verbatim.
\end{proof}

\begin{remark}\label{FunctorialityInsep}
The base change functor from the category of spherical embeddings of $X^0$ to the category of spherical embeddings of $X^0_{\kbar}$ (Remark~\ref{BCMorph}) is usually not an equivalence of categories (consider for instance $G = \GG_m$ and $X^0 = \GG_m$ with the action by left translation).

Let $A = \Aut_G(X^0)$ be the automorphism group scheme of the $G$-scheme $X^0$. As $A(k) \to A(\kbar)$ is injective and as all sets of morphisms between spherical embeddings of $X^0$ (resp.~of $X^0_{\kbar}$) are by definition subsets of $A(k)$ (resp.~of $A(\kbar)$) one sees that the base change functor is always faithful. 

It is an equivalence of categories if $A(k) \to A(\kbar)$ is bijective. This is for instance the case if $\Norm_G(H) = H$.
\end{remark}

On the other hand, by definition of morphisms of spherical embeddings we obtain the following equivalence of categories as a corollary of the proof of Theorem~\ref{ClassifySepClosed}. For this we define the category of colored fans for $(G,X^0)$ over a separably closed field as in Definition~\ref{DefMorphColored}.

\begin{corollary}\label{ClassConeInsep}
The construction $(X,i) \sends \Fscr(X,i) := (\Ccal_Y(X),\Fcal_Y(X))_Y$ (where $Y$ runs through the set of minimal $G$-invariant subschemes of $X$) in the proof of Theorem~\ref{ClassifySepClosed} yields an equivalence of the category of spherical embeddings of $X^0$ and the category of colored fans for $(G,X^0)$.
\end{corollary}

Note that as in Section~\ref{CLASSALG} the colored fan attached to $X$ depends on the choice of the Borel pair $(B,T)$ only up to unique isomorphism as for any two Borel pairs $(B,T)$ and $(B',T')$ over the separably closed field $k$ there exists an element $g \in G(k)$ with $B' = gBg^{-1}$ and $T' = gTg^{-1}$, and $g$ is unique up to right multiplication with an element in $T(k)$ (\cite[Exp.~XXVI~Lemme~1.16]{SGA3}).

\begin{proof}
We have to show that for two spherical embeddings $(X,i)$ and $(Y,j)$ of $X^0$ a $G$-equivariant morphism $\varphi^0\colon X^0 \to X^0$ extends to $X \to Y$ if and only if $\varphi^0$ is a morphism of the colored fans $\Fscr(X,i) \to \Fscr(Y,j)$. By Theorem~\ref{ClassSphericalAlg} it suffices to show that $(\varphi^0)_{\kbar}$ extends to $X_{\kbar} \to Y_{\kbar}$ if and only if $\varphi^0$ extends to $X \to Y$. But this is clear by faithfully flat descent because the extension is unique if it exists.
\end{proof}

%---------------------------------------------------------------------

\section{Galois descent for algebraic spaces with group actions}\label{DESCENT}

Before classifying spherical embeddings over arbitrary fields we recall some facts about Galois descent for algebraic spaces. All of them are probably well known but we could not find a good reference. Let $k$ be a field, let $k'$ be a Galois extension of $k$, and let $\Gamma := \Gal(k'/k)$. We also view $\Gamma$ as a projective limit $\Gamma_k$ of finite constant group schemes $\Gal(k'/k_1)_k$ over $k$, where $k_1$ runs through finite sub extensions of $k \subseteq k'$. Moreover we fix an algebraic group $G$ over $k$.

If $X$ is an algebraic space over $k$, $X' := X_{k'} = X \otimes_k k'$ is an algebraic space over $k'$ endowed with a $\Gamma$-action $a_X$ via the second factor which is compatible with the action on $k'$, i.e., for every $\gamma$ one has an automorphism $\gamma_{X}\colon X' \to X'$ making the diagram
\[\xymatrix{
X' \ar[r]^{\gamma_{X}} \ar[d] & X' \ar[d] \\
\Spec(k') \ar[r]^{\gamma} & \Spec(k')
}\]
commutative such that $\gamma_{X}\circ \delta_{X} = (\gamma\delta)_X$ for all $\gamma,\delta \in \Gamma$. We have the obvious notion of the category of algebraic spaces over $k'$ with compatible $\Gamma$-action. Moreover the action $a_X$ is \emph{continuous}, i.e., it yields a morphism of $k$-schemes $\Gamma_k \times X_{k'} \to X_{k'}$. The construction $X \sends (X_{k'},a_X)$ is functorial.

Now assume that $G$ acts on $X$. Then the functoriality shows that the action $a_X$ of $\Gamma$ on $X_{k'}$ is also $G$-compatible, i.e., for all $\gamma \in \Gamma$ the diagram
\[\xymatrix{
G_{k'} \times_{k'} X_{k'} \ar[r] \ar[d]_{\gamma_{G \times X}} & X_{k'} \ar[d]^{\gamma_X} \\
G_{k'} \times_{k'} X_{k'} \ar[r] & X_{k'}
}\]
is commutative, where the horizontal arrows are the action of $G_{k'}$ on $X_{k'}$.

\begin{proposition}\label{DescentSpace}
The functor $X \sends (X_{k'},a_X)$ is an equivalence of the category of algebraic spaces of finite type over $k$ with $G$-action with the category of algebraic spaces of finite type over $k'$ with $G_{k'}$-action endowed with a compatible and $G$-compatible continuous $\Gamma$-action.
\end{proposition}

\begin{proof}
By Remark~\ref{Approx} we may assume that $k'$ is a finite Galois extension of $k$. Then $\Spec k' \to \Spec k$ is faithfully flat of finite presentation and the result follows because fppf-descent data are effective for algebraic spaces (\cite[Tag~0ADV]{Stacks}).
\end{proof}

As usual one obtains a classification of forms by Galois cohomology groups: Let $X$ be an algebraic space of finite type over $k$ with $G$-action. Let $\Aut_G(X)(k')$ be the group of $G_{k'}$-equivariant automorphisms $X_{k'} \iso X_{k'}$ over $k'$. Then $\Gamma = \Gal(k'/k)$ acts on $\Aut_G(X)(k')$ by
\[
(\gamma,\alpha) \sends \gamma_X \circ \alpha \circ \gamma_X^{-1}, \qquad \gamma \in \Gamma, \alpha \in \Aut_G(X)(k')
\]
and this action factors through a quotient of $\Gamma$ by an open normal subgroup. We denote by $H^1(k'/k,\Aut_G(X)(k'))$ its first group cohomology. Let $E(k'/k,X)$ be the pointed set of isomorphism classes of forms of $X$, i.e. of algebraic spaces $X_1$ of finite type over $k$ with $G$-action such that there exists a $G$-equivariant isomorphism $X_{k'} \iso (X_1)_{k'}$. Then Proposition~\ref{DescentSpace} implies:

\begin{corollary}\label{ClassifyForm}
There is an isomorphism of pointed sets
\[
E(k'/k,X) \cong H^1(k'/k,\Aut_G(X)(k')).
\]
\end{corollary}

%=====================================================================

\section{Classification of spherical embeddings over arbitrary fields}

In this section let $k$ be an arbitrary field, let $k^s$ be a separable closure of $k$ and let $\kbar$ be an algebraic closure of $k^s$. Let $\Gamma = \Gal(k^s/k)$ be the Galois group of $k$. Whenever we speak of a continuous action of $\Gamma$ on a set $X$, we endow $X$ with the discrete topology and mean that the action map $\Gamma \times X \to X$ is continuous. If $X$ is finite and or if the action is linear on some finite-dimensional vector space, then the action is continuous if and only if it factors through some finite discrete quotient of $\Gamma$. Let $G$ be a reductive group over $k$.

Recall that every spherical $G$-space $X$ over $k$ is a spherical embedding of a spherical homogeneous $G$-space $X^0$ by Remark~\ref{SphericalEmbExist}. Moreover, $X^0$ is the unique open minimal $G$-invariant subspace $X^0$ of $X$. It is a smooth and quasi-projective scheme over $k$ by Lemma~\ref{HomogeneousScheme}.

Now fix a spherical homogeneous $G$-space $X^0$ over $k$. Then the Galois group $\Gamma$ acts linearly and continuously on $\Lambda_{(G_{k^s},X^0_{k^s})}$ and hence on the finite-dimensional $\QQ$-vector space $\Qcal := \Qcal_{(G_{k^s},X^0_{k^s})}$ continuously and linearly. Moreover it acts continuously on $\Vcal := \Vcal_{(G_{k^s},X^0_{k^s})}$ and the map $\Vcal \to \Qcal$ is $\Gamma$-equivariant. Finally, $\Gamma$ acts continuously on $\Dcal := \Dcal_{(G_{k^s},X^0_{k^s})}$ and the map $\rho\colon \Dcal \to \Qcal$ is $\Gamma$-equivariant. Here we use always that all construction depend on the choice of Borel pair of $G_{k_s}$ only up to a unique isomorphism.

We obtain the finite-dimensional $\QQ$-vector space $\Qcal$ with a continuous linear $\Gamma$-action, a subset $\Vcal \subseteq \Qcal$ that is $\Gamma$-stable, a finite set $\Dcal$ with a continuous $\Gamma$-action and a $\Gamma$-equivariant map $\rho\colon \Dcal \to \Qcal$.

If $(C,F)$ is a (strictly convex) colored cone for $(G_{k^s},X^0_{k^s})$, then $(\gamma(C),\gamma(F))$ is again a (strictly convex) colored cone for $(G_{k^s},X^0_{k^s})$ for all $\gamma \in \Gamma$. We call a colored fan $\Fscr$ for $(G_{k^s},X^0_{k^s})$ \emph{invariant under $\Gamma$} if for every $(C,F) \in \Fscr$ one has $(\gamma(C),\gamma(F)) \in \Fscr$ for all $\gamma \in \Gamma$.

Let $(X,i)$ be a spherical embedding of $X^0$. Then the constructions in $(X_{k^s},i_{k^s})$ yields a $\Gamma$-invariant colored fan $\Fscr(X,i)$ for $(G_{k^s},X^0_{k^s})$. Moreover if $(X,i)$ and $(X',i')$ are spherical embeddings of $X^0$ and $\varphi^0\colon X^0 \to X^0$ is a morphism $(X,i) \to (X',i')$ (i.e., $\varphi^0$ is $G$-equivariant and extends to a morphism $X \to X'$), then $\varphi^0_{k^s}$ is $\Gamma$-equivariant morphism of colored fans $\Fscr(X,i) \to \Fscr(X',i')$.

Altogether we obtain a functor $\Phi$ from the category of spherical embeddings of $X^0$ to the category of $\Gamma$-invariant colored fans for $(G_{k^s},X^0_{k^s})$.

\begin{theorem}\label{ClassField}
Let $X^0$ be a spherical homogeneous $G$-space over $k$. Then the following categories are equivalent.
\begin{equivlist}
\item
The category of spherical embeddings of $X^0$ (morphisms are morphisms $(X^0,X,i) \to (X^0,X',i')$ as defined in Definition~\ref{DefSphericalEmb}).
\item
The category of $\Gamma$-invariant colored fans for $(G_{k^s},X^0_{k^s})$.
\end{equivlist}
\end{theorem}

\begin{proof}
As explained in Section~\ref{DESCENT}, $(X^0)_{k^s}$ is naturally endowed with a continuous, compatible and $G$-compatible action by $\Gamma$. One easily checks (see also \cite[\S2]{Hur}) that Corollary~\ref{ClassConeInsep} induces an equivalence of the following categories.
\begin{equivlist}
\item
The category of $\Gamma$-compatible spherical embeddings of $(X_0)_{k^s}$, i.e. the category of pairs $(X',i')$, where $X'$ is a $G_{k^s}$-spherical space over $k^s$ endowed with a continuous, compatible and $G$-compatible action by $\Gamma$ and where $i\colon (X^0)_{k^s} \to X'$ is an open $G_{k^s}$-equivariant and $\Gamma$-equivariant open immersion.
\item
The category of $\Gamma$-invariant colored fans for $(G_{k^s},X^0_{k^s})$.
\end{equivlist}
Hence the theorem follows from Proposition~\ref{DescentSpace}.
\end{proof}

Huruguen has given in \cite{Hur} instructive examples of a reductive group $G$ over $\RR$, a spherical subgroup $H \subseteq G$ and a smooth spherical embedding of $(G/H)_{\CC}$ over $\CC$ whose colored fan is $\Gal(\CC/\RR)$-stable but which admits no $\RR$-form {\em as a scheme}. Hence the spherical embedding of $G/H$ attached to this colored fan by Theorem~\ref{ClassField} will be an algebraic space which is not a scheme.

Note that by \cite[Tag~0B88]{Stacks} a spherical space over $k$ is a scheme if any finite set of $\kbar$-rational points of $X_{\kbar}$ is contained in any open affine subscheme of $X_{\kbar}$. This yields Theorem~2.21 and Theorem~2.26 of \cite{Hur}.

%---------------------------------------------------------------------

\section{Classification of forms of spherical varieties}\label{CLFORM}

Let $k$ be a field, $\kbar$ an algebraic closure and $k^s$ the separable closure of $k$ in $\kbar$. Let $G$ be a reductive group over $k$, let $X$ be a $G$-spherical space over $k$, and let $X^0$ be the unique minimal $G$-invariant open subscheme. As $X^0$ is smooth over $k$, there exists a $G_{k_s}$-equivariant isomorphism $X^0 \iso G_{k_s}/H$ for some spherical subgroup $H$ of $G_{k^s}$.

The goal is to classify $k$-forms of $X$ in terms of the Galois cohomology of the automorphism group of $X$. We have by Corollary~\ref{ClassifyForm}:

\begin{corollary}\label{ClassifySpherForms}
There exists an isomorphism of pointed sets between isomorphism classes of $k^s/k$-forms of $X$ over $k$ and $H^1(k^s/k,\Aut_G(X)(k^s))$.
\end{corollary}

We conclude with some remarks on automorphisms of spherical spaces.

\begin{remark}\label{MorpHomogeneous}
Let $X^0 = G/H$ and $Y^0 = G/K$ for spherical subgroups $H$ and $K$ of $G$ such that the projection $G(k) \to (G/K)(k)$ is surjective (for instance if $k$ is an algebraically closed field). Then any $G$-equivariant morphism $\varphi^0\colon X^0 \to Y^0$ is induced by right translation with an element $g \in G(k)$ such that $H \subseteq gKg^{-1}$ and hence
\[
\Hom_G(G/H,G/K) = \set{g \in G(k)}{H \subseteq gKg^{-1}}/K,
\]
where $\Hom_G(\ ,\ )$ denotes the set of $G$-equivariant morphisms.
\end{remark}

For every $k$-scheme $T$ let $\Aut_G(X)(T)$ be the group of $G_T$-equivariant automorphisms of $X_T$. As $X^0$ is schematically dense in $X$ and $X$ is separated, the natural restriction $\Aut_G(X) \to \Aut_G(X^0)$ is a monomorphism. In particular, $\Aut_G(X)(k^s)$ will be a subgroup of $\Aut_G(X^0)(k^s)$.

One has $\Aut_G(X^0)(k^s) = \Aut_G((X^0)_{k^s}) = (\Norm_{G_{k^s}}(H)/H)(k^s)$.
In particular, we see that if $H = \Norm_{G_{k^s}}(H)$, then $\Aut_G(X)(k^s) = 1$ and hence there are no nontrivial $k'/k$-forms of $X$.

In \cite[Theorem~6.1]{Knop_LunaVust} it is shown that if $k$ is algebraically closed, then $\Aut_G(X^0)_{\red}$ is the extension of a diagonalizable group by a finite $p$-group. A simple corollary of the proof of loc.~cit.\ gives the following result.

\begin{proposition}\label{ShapeAuto}
There exists a largest unipotent subgroup scheme $U$ of $\Aut_{G}(X^0)$. It is finite, its formation is compatible with passing to field extensions $k \mono k'$, and one has an exact sequence of algebraic groups
\begin{equation}\label{EqShapeAuto}
1 \to U \to \Aut_G(X^0) \to M \to 1,
\end{equation}
where $M$ is a group of multiplicative type.
\end{proposition}

If ${\rm char}(k) = 0$, then every finite unipotent group is trivial and hence $\Aut_G(X^0) = M$.

\begin{proof}
By \cite[IV, \S2, Prop.~3.3]{DeGa} we can assume that $k$ is separably closed. Then $X^0 = G/H$ for a spherical subgroup $H$. Choose a Borel subgroup $B$ of $G$ such that $BH$ is open dense in $G$. Then the argument in \cite[Theorem~6.1]{Knop_LunaVust} shows that $\Aut_G(X^0)$ is a subquotient of $B$. In particular $\Aut_G(X^0)$ is trigonalizable and hence an extension of a diagonalizable group by the unique largest unipotent subgroup of $\Aut_G(X^0)$. To see that this largest unipotent subgroup is finite, we can pass to the algebraic closure of $k$ where we can again use \cite[Theorem~6.1]{Knop_LunaVust}.
\end{proof}

Over an algebraically closed field the sequence \eqref{EqShapeAuto} splits (\cite[Exp.~XVII, Th\'eor\`eme 5.1.1]{SGA3}). Hence if $k$ is perfect, Proposition~\ref{ShapeAuto} shows that there is an isomorphism of groups with $\Gamma$-action
\[
\Aut_G(X^0)(k^s) \cong A \ltimes M(k^s),
\]
where $A$ is a finite $p$-group on which $\Gamma$ acts. 

Note that for arbitrary separably closed fields the exact sequence \eqref{EqShapeAuto} might not yield an exact sequence on $k^s$-valued points (\cite[Exp.~XVII, contre-example 5.9(c)]{SGA3}).

%---------------------------------------------------------------------

\appendix

\section{Reminder on quotients}

We recall some results on quotients of group schemes that we use freely. If $G$ is a group scheme over a scheme $S$, then a subgroup scheme of $G$ is a group scheme $H$ over $S$ with a monomorphism of $S$-group schemes $H \to G$. If $S$ is the spectrum of a field, then the inclusion $H \to G$ is automatically a closed immersion (\cite[Exp.~VIA, Prop.~2.5.2]{SGA3}). We will always denote by $G/H$ the quotient in the category of fppf-sheaves on $S$. It is representable by an algebraic space if $H$ is flat and locally of finite presentation over $S$ (\cite[Tag~06PH]{Stacks}). In this case the canonical projection $G \to G/H$ is faithfully flat and locally of finite presentation.

If $G$ is flat (resp.~smooth, resp.~quasi-compact, resp.~quasi-separated, resp.~of finite presentation) over $S$, so is $G/H$: the projection $\pi\colon G \to G/H$ is faithfully flat and locally of finite presentation. Hence $G/H$ is clearly flat -- and even faithfully flat -- (resp.~smooth, resp.~quasi-compact) over $S$, if $G$ has the same property. If $G$ is quasi-separated (resp.~of finite presentation) over $S$, then $H$ is quasi-separated (resp.~of finite presentation) over $S$ and hence $\pi$ has the same property. Hence $G/H$ is quasi-separated over $S$ by \cite[(6.1.9)]{EGA1}\footnote{The proof generalizes verbatim to algebraic spaces.} (resp.~of finite presentation over $S$ by \cite[Tag~06NB]{Stacks}).

If $H$ is a closed subgroup scheme, then $G/H$ is separated: consider the 2-cartesian diagram of algebraic stacks
\[\xymatrix{
H \ar[rr] \ar[d] & & G \ar[d] \\
S \ar[r]^e & G \ar[r] & [H\backslash G/H]
}\]
As the projection $G \to [H\backslash G/H]$ is an atlas of the algebraic stack $[H\backslash G/H]$, we deduce that the lower horizontal map $S \to [H\backslash G/H]$ is representable by a closed immersion. Let $\mu\colon G/H \times G/H \to [H\backslash G/H]$ be the morphism induced by $(g_1,g_2) \sends g_2^{-1}g_1$. Then the 2-cartesian diagram
\[\xymatrix{
G/H \ar[r]^-{\Delta_{G/H}} \ar[d] & G/H \times G/H \ar[d]^{\mu} \\
S \ar[r] & [H\backslash G/H]
}\]
shows that $\Delta_{G/H}$ is a closed immersion.

If $S$ is locally noetherian of dimension $\leq 1$, $G$ is locally of finite type and $H$ is closed in $G$ and flat over $S$, then $G/H$ is representable by an $S$-scheme \cite[Th\'eor\`eme~4C]{Anan}.

%==================================================================

\end{document}